\documentclass[10pt]{article}
\usepackage{graphicx}      
\usepackage{natbib}        
\usepackage{amsfonts,amsmath}
\usepackage{amsthm}      
\usepackage{thmtools}    
\usepackage{dsfont}
\usepackage{color}
\usepackage{hyperref}
\usepackage{enumitem} 
\setlist[itemize,1]{leftmargin=\dimexpr 22pt}
\usepackage{cite}
\usepackage{microtype,enumerate,url}
\usepackage{BOONDOX-ds}
\usepackage{tikz} 
\usetikzlibrary{patterns}
\usepackage[fulladjust]{marginnote}
\usepackage{quoting}
\usepackage{caption}
\usepackage{subcaption}
\usepackage{diagbox}

\declaretheorem[style=definition]{theorem}

\declaretheorem[style=definition]{lemma}

\declaretheorem[style=definition,numbered=no]{standing assumption}
\declaretheorem[style=definition]{assumption}

\newcommand {\nn}{\nonumber}
\newcommand{\beq}{\begin{equation}}
\newcommand{\eeq}{\end{equation}}
\newcommand {\bseq}{\begin{subequations}}
\newcommand {\eseq}{\end{subequations}}
\newcommand {\bma}{\left[}
\newcommand {\ema}{\right]}

\newcommand {\Zplus}{\mathbb{Z}_{+}} 	
\newcommand {\R}{\mathbb{R}} 	
\newcommand {\Co}{\mathbb{C}} 	
\newcommand {\Cominus}{\mathbb{C}_{-}} 	
\newcommand {\Cozero}{\mathbb{C}_{0}} 	

\newcommand{\Image}{\operatorname{im}} 
\newcommand{\rank}{\operatorname{rank}} 
\newcommand{\transpose}{\mathsf{T}} 

\newcommand{\norm}[1]{\left\lVert#1\right\rVert}
\newcommand{\Span}{\operatorname{span}}

\newcommand{\spectrum}[1]{{\sigma({#1})}}
\newcommand{\overbar}[1]{\mkern 1.5mu\overline{\mkern-1.5mu#1\mkern-1.5mu}\mkern 1.5mu}

\begin{document}

\title{Least Squares Model Reduction:\\ 
A Two-Stage System-Theoretic Interpretation}

\author{Alberto Padoan\thanks{%
Department of Electrical and Computer Engineering, University of British Columbia, 2332 Main Mall, Vancouver, BC V6T 1Z4, Canada (e-mail: \url{alberto.padoan@ece.ubc.ca}).}}

\date{} 

\maketitle

\begin{abstract}
Model reduction simplifies complex dynamical systems while preserving essential properties. This paper revisits a recently proposed system-theoretic framework for least squares moment matching. It interprets least squares model reduction in terms of two steps process: constructing a surrogate model to satisfy interpolation constraints, then projecting it onto a reduced-order space. Using tools from output regulation theory and Krylov projections, this approach provides a new view on classical methods. For illustration, we reexamine the least-squares model reduction method by Lucas and Smith, offering new insights into its structure.
\end{abstract}


\section{Introduction} 

Model reduction simplifies dynamical systems while preserving their key features~\citep{antoulas2005approximation}. For linear time-invariant systems, moment matching has emerged as a popular approach~\citep{antoulas2005approximation,georgiou1983partial,kimura1986positive,antoulas1990solution,byrnes1995complete,georgiou1999interpolation,byrnes2001generalized,grimme1997krylov,gallivan2004sylvester,gallivan2006model,astolfi2010model}. It approximates a system's transfer function by a lower-degree function via rational interpolation, ensuring their moments — defined as coefficients of the Laurent series expansion — coincide at specified points. Efficient and robust implementations often use Krylov projectors~\citep[Chapter 11]{antoulas2005approximation}. 
Over the past three decades, moments of linear time-invariant systems have been characterized using Sylvester equations~\citep{gallivan2004sylvester,gallivan2006model} and, under specific assumptions, steady-state responses~\citep{astolfi2010model}.

However, classical moment matching methods enforce exact interpolation, which can be overly restrictive in practice. Least squares moment matching relaxes this requirement, optimizing mismatches in a least squares sense. For linear systems, it has a rich literature~\citep{shoji1985model,aguirre1992least,aguirre1994partial,smith1995least,aguirre1995algorithm,gugercin2006model,gu2010model,gustavsen1999rational,mayo2007framework,berljafa2017rkfit,nakatsukasa2018aaa,antoulas2020interpolatory}, with connections to Padé approximation~\citep{aguirre1992least,aguirre1994partial,aguirre1994model} and Prony's method~\citep{gugercin2006model,parks1987digital}. Least squares model reduction for linear and nonlinear systems is studied in~\citep{padoan2021model,padoan2023model}, building on the framework of~\citep{astolfi2010model}. Moments are characterized using output regulation theory~\citep{byrnes1989steady,isidori1990output,byrnes2000output,isidori1995nonlinear}, resulting in a time-domain formulation of least squares moment matching as a constrained optimization problem.   This formulation connects frequency-domain interpolation to system-theoretic concepts such as interconnections, invariance equations, and steady-state responses. 

\textit{Contributions}.
This paper makes two main contributions. First, we provide a new system-theoretic interpretation of least squares moment matching as a two-step model reduction process, bridging classical projection methods with recent output regulation insights. Second, we illustrate how this framework unifies and clarifies existing methods, particularly those based on moment matching at zero. For illustration, we analyze the approach of Smith and Lucas~\citep{smith1995least}, showing that this method yield models within the family
derived in~\citep{padoan2021model}.

\textit{Paper organization}. The rest of this work is structured as follows. Section~\ref{sec:problem_formulation} introduces the problem setup, including key definitions, assumptions, and the formulation of the moment matching problem. Section~\ref{sec:moment-matching} discusses moment matching for linear systems, focusing on methods using Sylvester equations and Krylov projections.
Section~\ref{sec:moment-matching} recalls the least-squares moment matching framework developed in~\citep{padoan2023model}.
Section~\ref{sec:system-theoretic-interpretation} explores the system-theoretic interpretation of least squares moment matching, presenting it as a two-step model reduction process, highlighting connections between our  least squares moment matching framework and existing methods. Section~\ref{sec:conclusion} concludes the paper with a summary and future directions.

\textit{Notation}.
$\Zplus$, $\R$ and $\Co$ denote the set of non-negative integers, the set of real numbers, and the set of complex numbers, respectively. 
$\Cominus$ and $\Cozero$ denote the sets of complex numbers with negative and zero real part, respectively.
$I$ denotes the identity matrix.
$\spectrum{A}$ denotes the spectrum of the matrix ${A \in \R^{n \times n}}$.
$M^{\transpose}$ denotes the transpose of the matrix  ${M \in \R^{p \times m}}$.
$\norm{\,\cdot\,}_2$ and $\norm{\,\cdot\,}_{2*}$ denote the Euclidean $2$-norm on $\R^{n}$ and the corresponding dual norm~\citep[p.637]{boyd2004convex}, respectively. 
$f^{(k)}(\cdot)$ denotes the derivative of order ${k\in\Zplus}$ of the function $f(\cdot)$, provided it exists, with $f^{(0)}(\cdot) = f(\cdot)$ by convention. $\lfloor x \rfloor$ denotes the largest integer less than or equal to ${x\in\R}$.

\section{Problem setup} \label{sec:problem_formulation}
 
Consider a continuous-time, single-input, single-output, 
linear, time-invariant system described by the equations
\beq \label{eq:system-linear}
\quad \dot{x} = Ax+Bu, \quad y=Cx,
\eeq
where ${x(t)\in\R^n}$, ${u(t)\in\R}$, ${y(t)\in\R}$, ${A\in\R^{n \times n}}$, ${B\in\R^{n \times 1}}$ and ${C\in\R^{1\times n}}$, with transfer function defined as 
\beq \label{eq:transfer-function} 
W(s)=C(sI-A)^{-1}B.
\eeq
For the notion of moment to make sense, we make the following standing assumption throughout the paper.

\begin{assumption} \label{ass:minimality}
The system~\eqref{eq:system-linear} is minimal, \textit{i.e.} the pair $(A,B)$ is controllable and the pair $(A,C)$ is observable.
\end{assumption}

\noindent
The \emph{moment of order\footnote{Moments may also be defined at \textit{poles} of the transfer function~\eqref{eq:transfer-function}. Following~\citep{padoan2017eigenvalues,padoan2017poles, padoan2017mrp,  padoan2019isolated1}, if ${s^{\star} \in \Co}$ is a pole of order ${m>0}$, the moment of order ${-k}$, with ${1 \le k \le m}$, is the coefficient of $(s-s^{\star})^{-k}$ in the Laurent series of~\eqref{eq:transfer-function} around~${s^{\star}}$.} ${k \in \Zplus}$} of system~\eqref{eq:system-linear} at ${s^{\star} \in \Co}$, with ${s^{\star} \not\in \spectrum{A}}$, is defined as
$${\eta_k(s^{\star}) = \tfrac{(-1)^k}{k!} W^{(k)}(s^{\star})}.$$

\noindent 
Given distinct \emph{interpolation points} ${\{s_i\}_{i=1}^N}$, with 
${s_i \in \Co}$ and ${s_i \not \in  \spectrum{A}}$,
and  corresponding  \emph{orders of interpolation}
${\{k_i\}_{i=1}^N}$, with ${k_i\in \Zplus}$,
model reduction by moment matching  consists in finding a system
\beq \label{eq:system-rom}
  \quad  \dot{\xi} = F\xi+Gv, \quad \psi = H\xi,
\eeq 
where ${\xi(t) \in\R^{r}}$, ${v(t)\in\R}$, ${\psi(t)\in\R}$, ${F\in\R^{r \times r}}$, ${G\in\R^{r \times 1}}$, ${H\in\R^{1\times r}}$, the transfer function of which
\beq \label{eq:transfer-function-model} 
\hat{W}(s)=H(sI-F)^{-1}G 
\eeq
satisfies the \emph{interpolation conditions}
\beq  \label{eq:matching-condition-linear}
\quad \eta_{j}(s_i)= \hat{\eta}_{j}(s_i) ,\quad j \in \{0, \ldots, k_i\}, \ i \in \{1, \ldots, N\},
\eeq 
where $\eta_{j}(s_i) $ and $ \hat{\eta}_{j}(s_i)$ denote the moments of order $j$  of~\eqref{eq:system-linear}   and~\eqref{eq:system-rom} at $s_i$, respectively. The system~\eqref{eq:system-rom} is said to be a \emph{model of system~\eqref{eq:system-linear}}   and to \emph{match the moments of system~\eqref{eq:system-linear} (or achieve moment matching) at $\{ s_i\}^{N}_{i=1}$} if the interpolation conditions~\eqref{eq:matching-condition-linear} hold.    Furthermore,~\eqref{eq:system-rom} is said to be a \emph{reduced order model of system~\eqref{eq:system-linear}} if ${r < n}$.

It is well-known that a model of order $r$ can match up to $2r$ moments~\citep[Chapter 11]{antoulas2005approximation}. The number of interpolation conditions is thus larger than the number of moments that can be matched if ${\nu = \sum_{i=1}^{N} (k_i+1)}> 2r$. In this case, the interpolation conditions~\eqref{eq:matching-condition-linear} give rise to an overdetermined system of equations which can be only solved in a {least squares sense}, leading directly to the \emph{model reduction problem by least squares moment matching}.

Given distinct interpolation points ${\{s_i\}_{i=1}^N}$, with 
${s_i \in \Co}$ and ${s_i \not \in  \spectrum{A}}$,
and corresponding  \emph{orders of interpolation} ${\{k_i\}_{i=1}^N}$, with ${k_i\in \Zplus}$,
least squares model reduction by moment matching  consists in finding a system~\eqref{eq:system-rom} which minimizes the index
\beq \label{eq:index}
\mathcal{J} = \sum_{i=1}^{N} \sum_{j=0}^{k_i} \left| \eta_{j}(s_i) - \hat{\eta}_{j}(s_i) \right|^2 .
\eeq 

The model~\eqref{eq:system-rom} is said to \emph{achieve least squares moment matching (at $\{ s_i\}^{N}_{i=1}$)} if it minimizes the index~\eqref{eq:index}.  For linear systems, both moment matching and least squares moment matching can be seen as a rational (Hermite) interpolation problem and,   thus, they are generically well-posed (see, e.g.,~\citep{gutknecht1990sense}).

\section{Moment matching} \label{sec:moment-matching}

Our analysis leverages standard moment matching methods, adapted
from~\citep{gallivan2004sylvester,gallivan2006model,astolfi2010model} and~\citep[Chapter 11]{antoulas2005approximation} with minor variations.

\subsection{Moment matching via Sylvester equations}

\begin{assumption} \label{ass:signal-generator-linear}
The matrix ${S\in \R^{\nu \times \nu}}$ is non-derogatory\footnote{A matrix is non-derogatory if its characteristic polynomials and its minimal polynomial coincide~\citep[p.178]{horn1994matrix}.} and has characteristic polynomial
\beq \label{eq:characteristic-polynomial}
\chi_S(s)=\prod_{i=1}^{N} (s-s_i)^{k_i+1},
\eeq 
with ${s_i \not \in  \spectrum{A}}$. The matrix ${L\in\R^{1\times \nu}}$ is such that the pair $(S,L)$ is observable.
\end{assumption}

\begin{lemma} \citep[Lemmas 3 and 4]{astolfi2010model}   \label{lemma:astolfi-3}
Consider system~\eqref{eq:system-linear}. Suppose Assumptions~\ref{ass:minimality} and~\ref{ass:signal-generator-linear} hold. Then there is a non-singular matrix ${T\in\R^{\nu\times \nu}}$ such that
\beq \nn
C\Pi T 
= 
[\, 
\eta_0(s_1) \cdots \eta_{k_1}(s_1) 
            \cdots 
\eta_0(s_N) \cdots \eta_{k_N}(s_N) 
\,],
\eeq
where  ${\Pi \in \R^{n\times \nu}}$ is the solution of the Sylvester equation
\beq \label{eq:Sylvester-equation-astolfi}
A\Pi +BL=\Pi S.
\eeq
\end{lemma} 

\noindent
Lemma~\ref{lemma:astolfi-3} establishes   that the moments of system~\eqref{eq:system-linear} can be \emph{equivalently} characterized  using  the Sylvester equation~\eqref{eq:Sylvester-equation-astolfi}.  This, in turn, is  key to characterize moments in terms of the \textit{steady-state behavior} of the interconnection of system~\eqref{eq:system-linear} with a signal generator described by the equations
\beq \label{eq:system-signal-generator}
 \dot{\omega} = S\omega, \quad  \theta = L\omega,  
\eeq 
with ${\omega(t) \in \R^\nu}$ and  ${\theta(t) \in \R},$ for which the following hold.  

\begin{theorem} \citep[Theorem 1]{astolfi2010model}   \label{thm:astolfi-1}
Consider system~\eqref{eq:system-linear} and the signal generator~\eqref{eq:system-signal-generator}.
Suppose Assumptions~\ref{ass:minimality} and~\ref{ass:signal-generator-linear} hold.
Assume ${\spectrum{A} \subset \Cominus}$ and ${\spectrum{S} \subset \Cozero}$. Then the steady-state output response of the interconnected system~\eqref{eq:system-linear}-\eqref{eq:system-signal-generator}, with ${u =  \theta}$, can be written as ${y_{\textup{ss}}(t) = C\Pi \omega(t)}$, with ${\Pi \in \R^{n\times \nu}}$ the solution of the Sylvester equation~\eqref{eq:Sylvester-equation-astolfi}.
\end{theorem}

\noindent
Theorem\,\ref{thm:astolfi-1}  leads a new   notion of model achieving moment matching. The system~\eqref{eq:system-rom} is a \emph{model of system~\eqref{eq:system-linear} at $(S,L)$}, with   ${S \in \R^{\nu\times \nu}}$   a non-derogatory matrix such that ${\spectrum{A} \cap \spectrum{S} = \emptyset}$, if ${\spectrum{F} 
\cap \spectrum{S} = \emptyset}$
and 
\beq \label{eq:model-condition-2}
C\Pi = HP,
\eeq
where ${\Pi \in \R^{ n \times \nu}}$ is the  solution of the Sylvester equation
\eqref{eq:Sylvester-equation-astolfi}, ${L \in \R^{1 \times \nu}}$ is such that the pair $(S,L)$ is observable, and ${P \in \R^{ r \times \nu}}$ is the  solution of the Sylvester equation
\beq \label{eq:Sylvester-equation-astolfi-model} 
FP +GL= P S.
\eeq
In this case, system~\eqref{eq:system-rom} is said to \emph{match the moment of system~\eqref{eq:system-linear} (or to achieve moment matching) at $(S,L)$}. Furthermore, system~\eqref{eq:system-rom} is a \emph{reduced order model of system~\eqref{eq:system-linear} at $(S,L)$} if ${r<n}$.

A family of  models  achieving moment matching for system~\eqref{eq:system-linear} has been defined in~\citep{astolfi2010model} by selecting $r=\nu$ and ${P=I}$ in~\eqref{eq:model-condition-2} and~\eqref{eq:Sylvester-equation-astolfi-model}, respectively.  This   yields a family of  models  at $(S,L)$ given  by~\eqref{eq:system-rom},  with
\beq \label{eq:family-linear-astolfi}
F=S-\Delta L, 
\quad G=\Delta, 
\quad  H = C \Pi,
\eeq   
with ${\Delta \in \R^{r \times 1}}$  such that ${\spectrum{S-\Delta L} \cap \spectrum{S}  = \emptyset}$. As discussed in~\citep{astolfi2010model}, the parameter ${\Delta}$ can be used to enforce prescribed properties, including stability, passivity, and a given $L_2$ gain.


\subsection{Moment matching via Krylov projections}

Moment matching via Krylov  projections is a numerically efficient way to construct models achieving moment matching~\citep[Chapter 11]{antoulas2005approximation}. Given system~\eqref{eq:system-linear}, the main idea is to define a model~\eqref{eq:system-rom}, by projection, as
\beq \label{eq:family-linear-Krylov}
F = P A Q,
\quad G=P B, 
\quad  H = C Q ,
\eeq
where ${P \in \R^{r \times n}}$ and ${Q \in \R^{n \times r}}$ are (biorthogonal Petrov-Galerkin) projectors such that $PQ=I$ and the interpolation conditions~\eqref{eq:matching-condition-linear} hold~\citep[Chapter 11]{antoulas2005approximation}.  The key tool to construct the projectors $P$ and $Q$ is the notion of \emph{Krylov subspace}. The {Krylov subspace} associated with a matrix ${M \in \Co^{n \times n}}$, a vector ${v\in\Co^n}$, a point ${s^{\star} \in \Co \cup \{ \infty\}}$, and a positive integer  $j>0$, is defined, for ${s^{\star} = \infty}$, as 
\beq \nn 
\mathcal{K}_j(M,v; s^{\star}) =     \Span \{ \, v , \,  \ldots , \, M^{j-1} v \,\}      , 
\eeq 
and,  for ${s^{\star} \not = \infty}$, as
\beq \nn
\mathcal{K}_j(M,v; s^{\star}) =      \Span \{  (s^{\star}I-M)^{-1} v  ,\, \ldots ,\, (s^{\star}I-M)^{-j} v \}     .
\eeq
The projectors are selected by ensuring that the image of $P^{\transpose}$ and $Q$ span the union of certain Krylov subspaces.

\begin{theorem}~\citep{grimme1997krylov} \label{thm:Krylov}
Consider system~\eqref{eq:system-linear}, 
 distinct interpolation points ${\{s_i\}_{i=1}^N}$, with 
${s_i \in \Co}$ and ${s_i \not \in  \spectrum{A}}$, and orders of interpolation ${\{k_i\}_{i=1}^N}$, with ${k_i\in \Zplus}$. Define the model~\eqref{eq:system-rom} as in~\eqref{eq:family-linear-Krylov}, where ${P \in \R^{r \times n}}$ and ${Q \in \R^{n \times r}}$ are such that ${PQ=I}$ and  ${\Image P^{\transpose} \supseteq \bigcup_{j =1}^{N} \mathcal{K}_{k_j}(A,B; s_j)}$. 
Then the model~\eqref{eq:system-rom} achieves moment matching at $\{ s_i\}^{N}_{i=1}$.
\end{theorem}

\section{Least squares moment matching}

Following~\citep{padoan2021model,padoan2023model}
models achieving least squares moment matching can be characterized, under certain assumptions, in terms of the solutions of the constrained optimization problem 
\beq \label{eq:optimization-problem}
\begin{array}{ll}
    \mbox{minimize}      & \norm{C \Pi - H P }_{2*}^2 \\
    \mbox{subject to}    & FP +GL= P S , \\
    						& \spectrum{S} \cap \spectrum{F} = \emptyset,
\end{array}
\eeq
where  ${F\in\R^{r \times r}}$, ${G\in\R^{r \times 1}}$, ${H\in\R^{1\times r}}$ and ${P\in\R^{r\times \nu}}$   are the optimization variables, and systems~\eqref{eq:system-linear} and~\eqref{eq:system-signal-generator}  
are the problem data. 
 
\begin{theorem} \label{thm:moments-optimization-linear}
Consider system~\eqref{eq:system-linear} and the signal generator~\eqref{eq:system-signal-generator}. Suppose Assumptions~\ref{ass:minimality} and~\ref{ass:signal-generator-linear} hold. Assume ${S+S^{\transpose} = 0}$.
Then the model~\eqref{eq:system-rom} achieves least squares moment matching at $\spectrum{S}$ if and only if there exists ${P\in\R^{r\times \nu}}$ such that $(F,G,H,P)$ is a solution of the optimization problem~\eqref{eq:optimization-problem}. 
\end{theorem}

\noindent
Theorem~\ref{thm:moments-optimization-linear} bears a number of interesting consequences. 
The connection between least squares moment matching and the optimization problem~\eqref{eq:optimization-problem} is instrumental to \textit{define} a least squares moment matching  for nonlinear systems~\citep{padoan2023model}. Furthermore, 
it allows one to characterize least squares moment matching as minimizing the worst case r.m.s. gain of  an error system~\citep[Theorem 3]{padoan2021model}. Finally, it also has implications for model reduction by least squares moment matching, as discussed next.

\subsection{Models achieving least squares moment matching} 

A family of models achieving least squares moment matching with a natural system-theoretic interpretation is obtained by treating $P$ as a fixed \emph{parameter}~\citep{padoan2021model}.

The family of models in question is described by the equations~\eqref{eq:system-rom}, with 
\beq \label{eq:family-linear}
F=P(S-\Delta L)Q,
\quad G=P \Delta , 
\quad  H = C \Pi Q,
\eeq
with ${S\in \R^{\nu \times \nu}}$ and ${L\in\R^{1\times \nu}}$ such that Assumption~\ref{ass:signal-generator-linear}
holds, ${\Pi \in \R^{n \times \nu}}$ the solution of the Sylvester equation~\eqref{eq:Sylvester-equation-astolfi},
and ${P \in \R^{r \times \nu}}$, ${\Delta \in\R^{\nu \times 1}}$ and ${Q \in\R^{\nu \times r}}$  such that
\begin{itemize}
\item[(i)]  the matrix $P$ is full rank and ${\ker P\subset \ker C\Pi}$,
\item[(ii)] the subspace $\ker P$ is an $(S,\Delta)$-controlled invariant\footnote{A subspace ${\mathcal{V}\subset \R^{n}}$ is an \textit{$(A, B)$-controlled invariant} if ${A\mathcal{V} \subset \mathcal{V} + \Span(B)}$~\citep[p.199]{basile1992controlled}.},
\item[(iii)]  ${PQ =I}$ and ${\spectrum{S} \cap \spectrum{P(S-\Delta L)Q} = \emptyset}$,
\end{itemize}
in which case $P$, $\Delta$ and $Q$ are said to be \emph{admissible}.

The main feature of the family of models~\eqref{eq:family-linear} is that it provides a solution to the least squares model reduction problem, as detailed by the following statement.  

\begin{theorem} \label{thm:1}
Consider system \eqref{eq:system-linear} and the family of models defined by~\eqref{eq:family-linear}. Suppose Assumptions~\ref{ass:minimality} and~\ref{ass:signal-generator-linear} hold. Assume ${S+S^{\transpose} = 0}$. Let ${P\in\R^{r\times \nu}}$,  ${Q\in\R^{\nu\times r}}$ and ${\Delta\in\R^{\nu \times 1}}$ be admissible for the parameterization \eqref{eq:family-linear}. Then the model \eqref{eq:family-linear} achieves least squares moment matching at $\spectrum{S}$.
\end{theorem}

\noindent
Theorem~\ref{thm:1} establishes that the  family of models~\eqref{eq:family-linear} achieves least squares moment matching (at ${\spectrum{S}}$).  The parameterization~\eqref{eq:family-linear} admits a natural system-theoretic interpretation in terms of a two-step model reduction process and allows one to relate parameters to properties of the model, as discussed in detail in the next section.

\section{A system-theoretic interpretation} \label{sec:system-theoretic-interpretation}

The family of models~\eqref{eq:family-linear} can be interpreted in system-theoretic terms as a two-step model reduction process which comprises the following basic steps. 
\begin{itemize}
\item[(I)] The first step consists in constructing a model which incorporates \emph{all} interpolation constraints~\eqref{eq:matching-condition-linear}. This leads to a model of order $\nu$, with $\nu > 2r$. The model is obtained using the parameterization~\eqref{eq:family-linear-astolfi} and is described by the matrices
\beq \label{eq:system-rom-surrogate}
\bar F= S-\Delta L ,
\quad \bar  G=\Delta , 
\quad  \bar  H = C \Pi , 
\eeq
with ${S\in \R^{\nu \times \nu}}$ and ${L\in\R^{1\times \nu}}$ such that Assumption~\ref{ass:signal-generator-linear}
holds, ${\Pi \in \R^{n \times \nu}}$ the solution of the Sylvester equation~\eqref{eq:Sylvester-equation-astolfi}, and ${\Delta \in \R^{\nu\times 1}}$ an admissible free parameter, respectively. 
\item[(II)] The second step consists in constructing a reduced order model of~\eqref{eq:system-rom-surrogate} through a (biorthogonal Petrov-Galerkin) projection, defined as
\beq \label{eq:system-rom-surrogate-reduced}
F=P \bar F Q,
\quad G=P \bar G, 
\quad  H = \bar H Q ,
\eeq
in which ${P \in \R^{r \times \nu}}$ and ${Q \in\R^{\nu \times r}}$  are admissible free parameters.
\end{itemize}
Note that~\eqref{eq:system-rom-surrogate} and~\eqref{eq:system-rom-surrogate-reduced} together yield~\eqref{eq:family-linear}, which shows that the family of models~\eqref{eq:family-linear} can be indeed described  as a two-step model reduction process. 
According to this interpretation, the family of models~\eqref{eq:family-linear} is obtained by approximation of an auxiliary, large model~\eqref{eq:system-rom-surrogate}  which takes into account \emph{all} interpolation constraints~\eqref{eq:matching-condition-linear}. For this reason,~\eqref{eq:system-rom-surrogate} is referred to as a \emph{surrogate model}. Fig.~\ref{fig:least-squares-moment-matching-linear-1} provides a diagrammatic illustration of  of least squares moment matching via surrogate models.

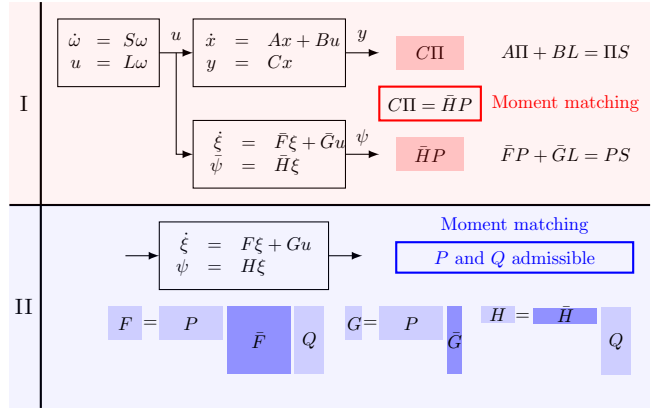
\begin{figure}[t!]
\centering
\begin{tikzpicture}[scale=0.45, every node/.style={scale=0.7}]


\node at (-2.5,3.5) {$\begin{array}{rcl}
\dot{ \omega} \!\!& = & \!\!  S  \omega\\
u \!\!& = & \!\!  L  \omega
\end{array}$
};
\node at (-0.5,4) {$u$};
\node at (2.3572,3.5) {
$\begin{array}{rcl}
\dot{x} & = & Ax+Bu  \\ y & = & Cx
\end{array}$};

\node at (5,4) {$y$};

\draw  (-4,4.5) rectangle (-1,2.5);
\draw  (0,4.5) rectangle (4.5,2.5);

\draw [-latex](-1,3.5) -- (0,3.5);
\draw [-latex](4.5,3.5) -- (5.5,3.5);


\node at (2.5,0.5) {
$\begin{array}{rcl}
\dot{\xi} & = & \bar{F}\xi+\bar{G}u  \\ \bar{\psi} & = & \bar{H}\xi
\end{array}$};

\node at (5,1) {$ \psi$};

\draw  (0,1.5) rectangle (4.5,-0.5);


\draw[red,thick]  (5.5,2.5) rectangle (8.5,1.5);
\node at (7,2) {$C\Pi = \bar{H}P$};
\node at (11,2) {\color{red} Moment matching};
 
\node at (7,3.5) {$C\Pi$};
\node at (11,3.5) {$A\Pi+BL=\Pi S$};
\fill[opacity=0.2,red]  (6,4) rectangle (8,3);
\node at (7,0.5) {$ \bar{H}P$};
\node at (11,0.5) {$ \bar{F}P+ \bar{G}L=P S$};
\fill[opacity=0.2,red]  (6,1) rectangle (8,0);

\draw [-latex](-0.5,3.5) -- (-0.5,0.5) -- (0,0.5);
\draw [-latex](4.5,0.5) -- (5.5,0.5);

\node at (-5,2) {\large  \textsc{I}};
\node at (-5,-4) {\large  \textsc{II}};

\node at (1.5,-2.5) {
$\begin{array}{rcl}
\dot{\xi} & = & F\xi+Gu  \\ \psi & = & H\xi
\end{array}$
}; 
\draw  (-1,-1.5) rectangle (4,-3.5);

\draw [-latex](-2,-2.5) -- (-1,-2.5);
\draw [-latex](4,-2.5) -- (5,-2.5);

\draw [ thick](-5.4284,-1) -- (13.5,-1);
\draw [ thick](-4.5,-7) -- (-4.5,5);

\node at (9.5,-1.6) {\color{blue} Moment matching};
\node at (9.5,-2.6) {\color{blue} $P$ and $Q$ admissible};


\fill[opacity=0.05,blue]  (-5.4284,-1) rectangle (13.5,-7);
\fill[opacity=0.05,red]  (-5.4284,-1) rectangle (13.5,5);


\fill[opacity=0.15,blue]  (-2.5,-4) rectangle (-1.5,-5);
\node at (-2,-4.5) {$F$};
\node at (-1.2668,-4.5) {$=$};
\fill[opacity=0.15,blue]  (-1,-4) rectangle (0.8666,-5);
\node at (-0.0667,-4.5) {$P$};
\fill[opacity=0.4,blue]  (1,-4) rectangle (2.8666,-6);
\node at (1.9333,-5) {$\bar F$};
\fill[opacity=0.15,blue]  (3,-4) rectangle (3.8666,-6);
\node at (3.4333,-5) {$Q$};
\fill[opacity=0.15,blue]  (4.5,-4) rectangle (5,-5);
\node at (4.7668,-4.5) {$G$};
\node at (5.2332,-4.5) {$=$};
\fill[opacity=0.15,blue]  (5.5,-4) rectangle (7.3666,-5);
\node at (6.4333,-4.5) {$P$};
\fill[opacity=0.4,blue]  (7.5,-4) rectangle (7.9333,-6);
\node at (7.7332,-5) {$\bar G$};
\fill[opacity=0.15,blue]  (8.5,-4) rectangle (9.5,-4.5);
\node at (9,-4.2332) {$H$};
\node at (9.7332,-4.2999) {$=$};

\fill[opacity=0.4,blue]  (10.0667,-4.0336) rectangle (11.9333,-4.5);
\node at (11,-4.2332) {$\bar H$};
\fill[opacity=0.15,blue]  (12.0667,-4.0336) rectangle (12.9333,-6.0336);
\node at (12.5,-5.0336) {$Q$};

\draw [blue, thick] (6,-2.1) rectangle (13,-3);
\draw [dotted, thick](-1.795,-6.1);

\node at (-5.5,5) {};
\node at (14,5) {};
\node at (14,-7) {};
\node at (-5.5,-7) {};

\end{tikzpicture}
\centering
\caption{{Least squares model reduction by moment matching as a two-step process using surrogate models.}}
\label{fig:least-squares-moment-matching-linear-1}
\end{figure}%
 
An alternative interpretation of the family of models~\eqref{eq:family-linear}  can be given in terms of a ``dual'' two-step model reduction process, which makes use of an auxiliary signal generator.
\begin{itemize}
\item[(I)] The first step consists in constructing reduced order model of the signal generator~\eqref{eq:system-signal-generator} through a (biorthogonal Petrov-Galerkin) projection. This step overcomes the issue of taking into account all interpolation constraints~\eqref{eq:matching-condition-linear} by approximating the signal generator~\eqref{eq:system-signal-generator}  with  a ``reduced'' signal generator of order $r$,  with ${r<\lfloor \nu / 2 \rfloor}$. The ``reduced'' signal generator in question is described by the matrices
\beq \label{eq:system-signal-generator-surrogate}
\bar S= P S Q ,
\quad \bar  L=  L Q ,
\eeq
with ${S\in \R^{\nu \times \nu}}$ and ${L\in\R^{1\times \nu}}$ such that Assumption~\ref{ass:signal-generator-linear}
holds, ${P \in \R^{r \times \nu}}$  and $Q \in \R^{\nu \times r}$ are admissible free parameters, respectively.

\item[(II)] The second step consists in building a model of the ``reduced'' signal generator using the parameterization~\eqref{eq:family-linear-astolfi}, defined as
\beq \label{eq:system-signal-generator-surrogate-reduced}
F = \bar  S-  \bar \Delta \bar L ,
\quad G=\bar  \Delta , 
\quad  \bar  H = C \bar \Pi .
\eeq 
where ${\bar \Pi \in \R^{n \times r }}$ is the solution of the ``reduced'' Sylvester equation
\beq \label{eq:Sylvester-equation-surrogate}
A\bar\Pi +B \bar L=\bar \Pi \bar S,
\eeq
and $ {\Delta \in \R^{r \times 1}}$ such that ${\spectrum{\bar S} \cap \spectrum{\bar  S-  \bar \Delta \bar L } = \emptyset}$. Since $$\rank P = r,$$ the vector $\bar \Delta$ 
can be written  as 
$${\bar \Delta = P \Delta},$$ 
which shows that the family~\eqref{eq:system-signal-generator-surrogate-reduced} can be equivalently defined as
\beq \label{eq:system-signal-generator-surrogate-reduced-equivalent}
F = \bar  S-  P \Delta \bar L ,
\quad G=P \Delta, 
\quad  \bar  H = C \bar \Pi .
\eeq 
\end{itemize}
Note that~\eqref{eq:system-signal-generator-surrogate} and~\eqref{eq:system-signal-generator-surrogate-reduced} together yield~\eqref{eq:family-linear}, which shows that the family of models~\eqref{eq:family-linear} can be indeed described  as a two-step model reduction process. 
According to this interpretation, the family of models~\eqref{eq:family-linear} is obtained by approximation of an auxiliary (possibly large) signal generator~\eqref{eq:system-signal-generator-surrogate} which is then used to construct a reduced order model. For this reason,~\eqref{eq:system-signal-generator-surrogate} is referred to as a \emph{surrogate signal generator}. Fig.~\ref{fig:least-squares-moment-matching-linear-2} provides a diagrammatic illustration of least squares moment matching via surrogate signal generators.

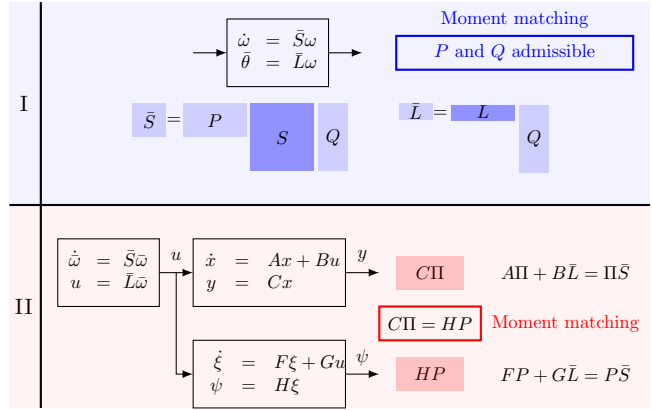
\begin{figure}[t!]
\centering
\begin{tikzpicture}[scale=0.45, every node/.style={scale=0.7}]


\node at (-2.5,-3) {$\begin{array}{rcl}
\dot{\bar \omega} \!\!& = & \!\! \bar S \bar \omega\\
u \!\!& = & \!\! \bar L \bar \omega
\end{array}$
};
\node at (-0.5,-2.5) {$u$};
\node at (2.3572,-3) {
$\begin{array}{rcl}
\dot{x} & = & Ax+Bu  \\ y & = & Cx
\end{array}$};

\node at (5,-2.5) {$y$};

\draw  (-4,-2) rectangle (-1,-4);
\draw  (0,-2) rectangle (4.5,-4);

\draw [-latex](-1,-3) -- (0,-3);
\draw [-latex](4.5,-3) -- (5.5,-3);


\node at (2.5,-6) {
$\begin{array}{rcl}
\dot{ \xi} & = &  F \xi+ Gu  \\ \psi & = &  H \xi
\end{array}$};

\node at (5,-5.5) {$ \psi$};

\draw  (0,-5) rectangle (4.5,-7);


\draw[red,thick]  (5.5,-4) rectangle (8.5,-5);
\node at (7,-4.5) {$C\Pi = HP$};
\node at (11,-4.5) {\color{red} Moment matching};
 
\node at (7,-3) {$C\Pi$};
\node at (11,-3) {$A\Pi+B \bar L=\Pi \bar S$};
\fill[opacity=0.2,red]  (6,-2.5) rectangle (8,-3.5);
\node at (7,-6) {$ HP$};
\node at (11,-6) {$ FP+ G \bar L=P \bar S$};
\fill[opacity=0.2,red]  (6,-5.5) rectangle (8,-6.5);

\draw [-latex](-0.5,-3) -- (-0.5,-6) -- (0,-6);
\draw [-latex](4.5,-6) -- (5.5,-6);

\node at (-5,2) {\large  \textsc{I}};
\node at (-5,-4) {\large  \textsc{II}};

\node at (2.5,3.5) {
$\begin{array}{rcl}
\dot{ \omega} \!\!& = & \!\! \bar S \omega\\
\bar \theta \!\!& = & \!\! \bar L \omega
\end{array}$
}; 
\draw  (1,4.5) rectangle (4,2.5);

\draw [-latex](0,3.5) -- (1,3.5);
\draw [-latex](4,3.5) -- (5,3.5);

\draw [ thick](-5.4284,-1) -- (13.5,-1);
\draw [ thick](-4.5,-7) -- (-4.5,5);

\node at (9.5,4.5) {\color{blue}  Moment matching};
\node at (9.5,3.5) {\color{blue} $P$ and $Q$  admissible};


\fill[opacity=0.05,red]  (-5.4284,-1) rectangle (13.5,-7);
\fill[opacity=0.05,blue]  (-5.4284,-1) rectangle (13.5,5);


\fill[opacity=0.15,blue]  (-1.795,2) rectangle (-0.795,1);
\node at (-1.295,1.5) {$\bar S$};
\node at (-0.5618,1.5) {$=$};
\fill[opacity=0.15,blue]  (-0.295,2) rectangle (1.5716,1);
\node at (0.6383,1.5) {$P$};
\fill[opacity=0.4,blue]  (1.705,2) rectangle (3.5716,0);
\node at (2.6383,1) {$S$};
\fill[opacity=0.15,blue]  (3.705,2) rectangle (4.5716,0);
\node at (4.1383,1) {$Q$};
\fill[opacity=0.15,blue]  (6.0716,2) rectangle (7.0716,1.5);
\node at (6.5716,1.7668) {$\bar L$};
\node at (7.3048,1.7001) {$=$};

\fill[opacity=0.4,blue]  (7.6383,1.9664) rectangle (9.5049,1.5);
\node at (8.5716,1.7668) {$ L$};
\fill[opacity=0.15,blue]  (9.6383,1.9664) rectangle (10.5049,-0.0336);
\node at (10.0716,0.9664) {$Q$};

\draw [blue, thick] (6,4) rectangle (13,3);
\draw [dotted, thick](-1.795,0);

\node at (-5.5,5) {};
\node at (14,5) {};
\node at (14,-7) {};
\node at (-5.5,-7) {};

\end{tikzpicture}
\centering
\caption{{Least squares model reduction by moment matching as a two-step process using surrogate signal generators.}}
\label{fig:least-squares-moment-matching-linear-2}
\end{figure}%

\subsection{Connections with the literature}

This section highlights connections between model reduction by least squares moment matching and existing methods based on moment matching and least squares. The family~\eqref{eq:family-linear} encompasses a broad class of models derived from these methods, as summarized in Table~\ref{table:connections}.


\begin{table*}[h!]
\centering
\renewcommand{\arraystretch}{1.5}
\resizebox{1\textwidth}{!}{%
\begin{tabular}{|c|c|c|c|}
\hline
\diagbox{\scalebox{0.75}{Parameter}}{\scalebox{0.75}{Method}} & Sylvester equations & Krylov projections & Smith--Lucas method \\ \hline
$\Pi$ & given & $I$ & given \\ \hline
$\Delta$ & $\spectrum{S} \cap \spectrum{S-\Delta L} = \emptyset$ & $B$ & $\chi_{\overbar S}(s)=s^{r}$ \\ \hline
$P$ & $I$ & $\ker P \subset \ker C$, $(S,L)$-invariant & $\chi_{\overbar S}(s)=s^{r}$ \\ \hline
$Q$ & $I$ & $PQ=I$ & $\chi_{\overbar S}(s)=s^{r}$ \\ \hline
$\nu$ & $r$ & $n$ & $r+q$ \\ \hline
\end{tabular}
}
\vspace{.15cm}
\caption{Connections between the parameterization~\eqref{eq:family-linear} and models obtained via existing moment matching and least squares moment matching methods.}
\label{table:connections}
\end{table*}

\subsection{Moment matching via Sylvester equations}
Any model of the form~\eqref{eq:family-linear-astolfi} is a special case of the family~\eqref{eq:family-linear}, where the surrogate model~\eqref{eq:system-signal-generator-surrogate} \textit{coincides} with the signal generator~\eqref{eq:system-signal-generator}. Setting $P = I$ and $Q = I$ produces~\eqref{eq:family-linear-astolfi}, with $\Delta \in \R^{r \times 1}$ satisfying $\spectrum{S} \cap \spectrum{S - \Delta L} = \emptyset$. Admissibility follows by design, since $\ker{P} = \{0\}$ and $\spectrum{S} \cap \spectrum{S - \Delta L} = \emptyset$.

\subsection{Moment matching via Krylov projections}
Any model of the form~\eqref{eq:family-linear-Krylov} is a special case of the family~\eqref{eq:family-linear}, where the surrogate model~\eqref{eq:system-rom-surrogate} \textit{coincides} with the system~\eqref{eq:system-linear}. Setting ${\Pi = I}$ and ${\Delta = B}$ gives ${A = S - BL}$, which directly yields~\eqref{eq:family-linear-Krylov}. Admissibility of $P$, $Q$, and $\Delta$ reduces to ensuring that ${\ker P}$ is an $(S, \Delta)$-controlled invariant subspace contained in ${\ker C}$ and that $PQ = I$. If $P$ and $Q$ are constructed as in Theorem~\ref{thm:Krylov}, these conditions always hold.

\subsection{Least squares moment matching method at zero }

Model reduction by least squares moment matching at zero has been studied extensively~\citep{shoji1985model,aguirre1992least,aguirre1994partial,aguirre1994model,smith1995least,aguirre1995algorithm}. This section highlights that these methods yield models within the family~\eqref{eq:family-linear}. For illustration, we focus on the approach proposed by Smith and Lucas~\citep{smith1995least}.

Consider system~\eqref{eq:system-linear} and the model~\eqref{eq:system-rom}. Assume $0 \not \in \spectrum{A}$. The Laurent series expansion of the transfer function~\eqref{eq:transfer-function} of system~\eqref{eq:system-linear}  around zero is
\beq \label{eq:transfer-function-system-smith}  
W(s) = \eta_{0}(0) + \eta_{1}(0)s + \eta_{2}(0)\frac{s^2}{2!} + \ldots \, .  
\eeq  
The Smith-Lucas method seeks to determine the coefficients of the transfer function~\eqref{eq:transfer-function-model} of the model~\eqref{eq:system-rom}, defined as
\beq \label{eq:transfer-function-rom-smith}
\hat{W}(s) = 
\frac{ \hat \beta_{r-1} s^{r-1} + \ldots +\hat \beta_0
}{ s^{r} + \hat \alpha_{r-1} s^{r-1} + \ldots +\hat \alpha_0} ,
\eeq
by enforcing the interpolation conditions
\beq  \label{eq:matching-condition-linear-smith}
\eta_{j}(0)= \hat{\eta}_{j}(0),  \qquad 0 \le j \le 2r+q-1,
\eeq 
where $q$ is a given positive integer. 
Following~\citep{smith1995least}, this amounts to solving the equation
\beq  \label{eq:matching-condition-system-smith}
\bma 
\begin{array}{cc}
0 & X \\
I & Y
\end{array}
\ema 
\bma
\begin{array}{c}
\hat \alpha \\ 
\hat \beta
\end{array}
\ema  
= 
\bma
\begin{array}{c}
\mu \\
0
\end{array}
\ema ,
\eeq
in the unknowns $\hat \alpha \in \R^r$ and $\hat \beta \in \R^r$, where
\bseq
\begin{align}
\hat \alpha 
&= \bma \begin{array}{ccc} \hat \alpha _0  & \ldots & \hat \alpha_{r-1} \end{array}\ema^{\transpose} , \\
\hat \beta 
&= \bma \begin{array}{ccc} \hat \beta _0  & \ldots & \hat \beta_{r-1} \end{array} \ema^{\transpose} , \\
\mu
&= \bma \begin{array}{ccc} \eta_0(0)  & \ldots & \eta_{r+q-1}(0) \end{array} \ema^{\transpose}  , \\
X
&=  -
\bma 
\begin{array}{cccc} 
\eta_r(0)  & \eta_{r-1}(0) & \ldots & \eta_{1}(0) \\
\eta_{r+1}(0)  & \eta_{r}(0) & \ldots & \eta_{2}(0) \\
\vdots &  \vdots & \ddots & \vdots \\
\eta_{2r+q-1}(0)  & \eta_{2r+q-2}(0) & \ldots & \eta_{r+q}(0) 
\end{array} 
\ema  , \\
Y
&=  -
\bma 
\begin{array}{cccc} 
\eta_0(0)  & 0 & \ldots & 0 \\
\eta_{1}(0)  & \eta_{0}(0) & \ddots & \vdots  \\
\vdots &  \vdots & \ddots & 0 \\
\eta_{r-1}(0)  & \eta_{r-2}(0) & \ldots & \eta_{0}(0) 
\end{array} 
\ema  .
\end{align}
\eseq
Assuming $\eta_0(0)\not=0$ for simplicity and solving~\eqref{eq:matching-condition-system-smith}, in turn, amounts to solving the system of equations
\bseq
\begin{align}
&X^{\transpose}  X \hat \alpha = X^{\transpose} \mu,  \label{eq:LS-equation-alpha} \\
&\hat \alpha + Y  \hat \beta = 0, \label{eq:LS-equation-beta} 
\end{align}
\eseq
the unique solution of which is ${\hat \alpha  =  (X^{\transpose}  X)^{-1} X^{\transpose} \mu}$ and ${\hat \beta =  -Y^{-1} \hat \alpha.}$
Adapting notation and terminology to the present context, a particularly interesting conclusion of~\citep{smith1995least} is
\begin{quoting}[
     indentfirst=true,
     leftmargin=\parindent,
     rightmargin=\parindent] {\itshape
``... the technique may now be thought of as the two-step process of denominator calculation, by solving~\eqref{eq:LS-equation-alpha}  in a least-squares sense, and then numerator calculation by simple substitution into~\eqref{eq:LS-equation-beta}. In solving~\eqref{eq:LS-equation-alpha} it is seen that the Euclidean norm of the vector $X \alpha - \mu$ is minimized, and $\mu$ contains the first $r + q$ moments [at zero] of the system, whereas $X \alpha$ is an estimate of the moments [at zero] of the corresponding reduced model because $X$, contains the ... moment parameters of the full system and not those of the reduced model. Hence, it is clear that the method minimizes the index 
\beq \label{eq:index-LS}
\mathcal{J} = \sum_{j=0}^{r+q-1} \left| \eta_{j}(0) - \hat{\eta}_{j}(0) \right|^2 
\eeq
for the denominator, where $\hat{\eta}_{j}(0)$, with $0 \le i \le r+q-1$, are estimates of the moments [at zero] of the reduced model (contained in $X\alpha$), and then matches the first $r$ moments [at zero] \textit{exactly} for the numerator coefficients.'' }
\end{quoting}

We conclude that the method presented in~\citep{smith1995least} boils down to a two-step model reduction process, in which the first $r$ moments at zero are matched \textit{exactly} and the first $r+q$ moments at zero are matched \textit{in a least squares sense}.  Consequently, any model obtained via the Smith-Lucas method belongs to the family of models~\eqref{eq:family-linear} and corresponds to the special case in which the characteristic polynomial of the signal generator~\eqref{eq:system-signal-generator} is
\beq \label{eq:characteristic-polynomial-LS}
\chi_{S}(s)= s^{r+q}
\eeq
and that of the surrogate signal generator~\eqref{eq:system-signal-generator-surrogate} is 
\beq \label{eq:characteristic-polynomial-LS-surrogate}
 \chi_{\overbar S}(s)=s^{r} ,
\eeq
This property can be enforced by selecting ${P}$, $\Delta$ and $Q$ such that ${SQ = Q \bar S}$, as detailed by the following result.

\begin{theorem} \label{thm:preservation-interpolation-points}
Consider system~\eqref{eq:system-linear} and the family of models defined by~\eqref{eq:family-linear}. Suppose Assumptions~\ref{ass:minimality} and~\ref{ass:signal-generator-linear} hold. Let ${P \in \R^{r \times \nu}}$,  ${\Delta \in\R^{\nu \times 1}}$ and ${Q \in \R^{\nu \times r}}$  be admissible for the parameterization~\eqref{eq:family-linear} and 
${\bar{S}\in \R^{r \times r}}$ such that ${SQ = Q \bar S}$. Then the model \eqref{eq:family-linear} achieves moment matching at zero (up to the order ${r}$) and least squares moment matching at zero (up to the order ${r+q}$).
\end{theorem}

\begin{proof}
Theorem~\ref{thm:moments-optimization-linear} directly implies that the model achieves least squares moment matching at zero up to the order ${r+q}$. Moreover, the model~\eqref{eq:family-linear} is such that the interpolation conditions 
\beq  \label{eq:matching-condition-linear-smith-explained}
\eta_{j}(0)= \hat{\eta}_{j}(0),  \qquad 0 \le j \le r,
\eeq 
hold, since the corresponding surrogate signal generator~\eqref{eq:system-signal-generator-surrogate}
has characteristic polynomial~\eqref{eq:characteristic-polynomial-LS-surrogate}. Since by assumption ${SQ = Q \bar S}$, we conclude that ${\spectrum{\bar S} \subset \spectrum{S}}$ and, hence, that the model~\eqref{eq:family-linear} achieves moment matching at zero up to the order ${r}$.
\end{proof}

\noindent
Theorem~\ref{thm:preservation-interpolation-points} provides a new interpretation which places the Smith-Lucas method on firm system-theoretic footing, showing it fits naturally within the two-step surrogate framework introduced in this paper.

\section{Conclusion} \label{sec:conclusion}
This paper has revisited least squares moment matching, showing that under specific assumptions, the process has been interpreted as a two-step model reduction. The first step constructes a surrogate model to satisfy interpolation constraints; the second step computes a reduced-order model by projection. This reinterpretation has provided new insights into the structure of the least-squares model reduction method by Lucas and Smith, revealing its underlying system-theoretic structure. Future work could investigate extending our findings to nonlinear systems, leveraging results from~\citep{astolfi2010model} and to time-varying systems using insights from~\citep{scarciotti2015modelexplicit}.


\end{document}